\documentclass[11pt]{amsart}
\usepackage{latexsym,amsmath,amssymb,amscd,cleveref,enumerate}
\usepackage[dvips]{color}

\begin{document}

\title{Existence of Rotating magnetic stars}
\author{Juhi Jang}
\address{Department of Mathematics, University of Southern California, Los Angeles, CA 90089, USA and Korea Institute for Advanced Study, Seoul, Korea}
\author{Walter A. Strauss} 
\address{Department of Mathematics and Lefschetz Center for Dynamical Systems, Brown University, Providence, RI 02912}
\author{Yilun Wu}
\address{Department of Mathematics, University of Oklahoma, Norman, OK 73069}
\date{\today}
\maketitle

\newtheorem{corollary}{Corollary}
\newtheorem{proposition}{Proposition}
\newtheorem{theorem}{Theorem}[section]
\newtheorem{lemma}{Lemma}[section]
\newtheorem*{exer}{Exercise}
\newtheorem{prop}{Proposition}[section]
\newtheorem{remark}{Remark}
\newtheorem{assump}{Assumption}
\newtheorem*{claim}{Claim}

\newcommand{\real}{\mathbb{R}}
\newcommand{\gz}{g_{\zeta}}
\newcommand{\gzone}{g_{\zeta_1}}
\newcommand{\gztwo}{g_{\zeta_2}}
\newcommand{\tr}{\textnormal{ tr}}
\newcommand{\vel}{\mathbf{v}}

\newcommand{\e}{\varepsilon}
\def\eqn {\begin{equation}}
\def\eeqn {\end{equation}}
\def\C{{\mathbb C}}
\def\real{{\mathbb R}}
\def\R{\real}
\def\Z{{\mathbb Z}}
\def\g{\kappa}  
\def\ep{\epsilon}
\def\th{\theta}
\def\lb{\lambda}
\def\al{\alpha}
\def\va{\varphi}
\def\pa{\partial}
\def\nb{\nabla}
\def\A{\mathcal A}
\def\B{\mathcal B}
\def\C{\mathcal C}
\def\H{\mathcal H}
\def\J{\mathcal J}
\def\G{\mathcal G}
\def\F{\mathcal F}
\def\K{\mathcal K}
\def\L{\mathcal L}
\def\M{\mathcal M}
\def\O{\mathcal O}
\def\P{\mathcal P}
\def\S{\mathcal S}
\def\W{\mathcal W}
\def\ka{\kappa}
\def\LL{\langle}
\def\RR{\rangle}

\begin{abstract}
We consider a star as a compressible fluid subject to gravitational and magnetic forces.  
This leads to an Euler-Poisson system coupled to a magnetic field,  
which may be regarded as an MHD model together with gravity. 
The star executes steadily rotating motion about a fixed axis.  
We prove, for the first time, the existence of such stars provided that the rotation speed 
and the magnetic field are sufficiently small. 
\end{abstract}

\section{Introduction} 

There have been extensive mathematical studies of stars subject only to gravitational forces but very 
few that incorporate magnetic forces.  
The only study that we are aware of is \cite{FLS}, in which the star does not rotate.  
It is well known that magnetic forces have major physical effects, for instance in the reconnection 
phenomenon of solar flares.  
Stellar magnetism is a very active area of physical theory \cite{Mestel, Shapiro}, typically modeled by MHD,  
as well as of observation \cite{GST}.  
Because it is rare for stars to have a net charge, it is frequently assumed that the electric field vanishes.  

Our model consists of the steady compressible Euler equations together with gravity and magnetic terms.  
It is as follows.  
\eqn\label{masscons}   \nb\cdot(\rho v)=0  \eeqn
\eqn \label{eq: 2}  \rho (v\cdot\nb) v + \nb p = \rho\nb U + (\nb\times B)\times B  \eeqn
\eqn\label{eq: 3}   \nb\times(v\times B)=0   \eeqn
\eqn\label{eq: 4}   \nb\cdot B = 0  \eeqn
\eqn \label{eq: 5}  \Delta U = -4\pi\rho   \eeqn 
Equations \eqref{masscons}-\eqref{eq: 3} should hold in the fluid domain $\{\rho>0\}$, while \eqref{eq: 4} 
and \eqref{eq: 5} should hold in all of $\mathbb R^3$. For simplicity, the 
magnetic permeability is set equal to 1 throughout $\mathbb R^3$, although 
more realistically it could differ inside and outside the star. 
We further impose the boundary conditions 
$\lim_{|x|\to\infty}U(x)=0$, an equation of state $p=p(\rho)$, and 
\eqn\label{cond: vac bdry}
p=0 \text{ on the set }\pa\{\rho>0\}.
\eeqn

The first two equations express mass and momentum conservation.  
The magnetic force is $J\times B$, where $J = \nb\times B$ 
(from Amp\`ere's Law in Maxwell's equations without $E$) 
is the magnetic current, omitting the usual $4\pi$ factor.  
The third equation comes from Faraday's Law in Maxwell's equations, where 
the electric field and the conductivity have been neglected due to the large length scale in astrophysics.  
The fourth equation is one of Maxwell's equations and the fifth is gravity.  

We assume a steady rotation $v = \omega r e_\theta$ around the $x_3$ axis, 
where $\omega$ is a constant rotation speed 
and $e_\theta = (-\sin\theta, \cos\theta, 0)$ in cylindrical coordinates.  
Then \eqref{masscons} is satisfied and $(v\cdot\nb) v = -\omega^2 \nb(r^2)/2$ where $r^2=x_1^2+x_2^2$.  
Furthermore an equation of state is assumed: $p$ is a function of $\rho$.  
For instance, we allow $p(\rho) = \rho^\gamma$ with $\frac65<\gamma<2, \gamma\ne\frac43$.  
The specific enthalpy is defined as 
\eqn \label{enthalpy} h(\rho) = \int_0^\rho \frac{p'(s)}{s}\ ds.  \eeqn  
We will show in Section 2 that, due to the cylindrical symmetry, there is a scalar function $\psi$ 
such that $rB^r = \pa_{3}\psi,  \ rB^3 = -\pa_r\psi$.  

Under these conditions we will show in Section 2 that the system reduces to the three equations   
\eqn \label{7}   -\frac12 \omega^2 r^2 +h(\rho) - U + \ep K(\psi) = \text{ constant\quad in } \{\rho>0\},  \eeqn
\eqn \label{8}   L\psi = \ep k(\psi) \rho \quad\text{ in } \real^3,  \eeqn
\eqn  \label{9}  U = |x|^{-1} * \rho  \quad\text{ in } \real^3,  \eeqn
where $L = \nb\cdot r^{-2}\nb$ and where $k=K'$ is an arbitrary function.  
We call $k$ the magnetic current function, because it takes the magnetic potential $\psi$ 
to a multiple of the magnetic current $J=\nb\times B$ (see Section 2).

We will prove the existence of solutions by a perturbation analysis starting from a spherically symmetric 
stationary solution.  
The tool we employ is the standard implicit function theorem in Banach space.  

For any radius $R>0$, it is well-known \cite {SW2017} that there exists a unique spherical 
solution  $\rho_0(|x|)\ge0$, called the Lane-Emden solution, 
with $\omega=\ep=0$ and $\psi\equiv0$ such that 
$\rho_0>0$ in $B_R=\{|x|<R\}$ and $\rho_0\in C^2(B_R)\cap C^{1,\alpha}(\real^3)$, 
where $\alpha=\min\left(\frac{2-\gamma}{\gamma-1},1\right)$.  
Our main theorem is as follows.  

\begin{theorem}
Let $R>0$.  Let $\rho_0$ be the unique solution mentioned above.  Let $k\in C^2(\real)$.  
Let $p(\rho)=\rho^\gamma$ where $6/5 < \gamma < 2$ and $\gamma\ne 4/3$.     
Then there exist $\bar\omega>0$ and $\bar\ep>0$ 
and solutions $(\rho=\rho_{\omega,\ep}, \psi=\psi_{\omega,\ep})$ 
for all $|\omega|<\bar\omega$ and $|\ep|<\bar\ep$, 
with the following properties.  

$\bullet$  $\rho\in C_c^{1,\alpha}(\real^3), \psi\in C^{3,\alpha}(\real^3) $, where $\alpha = \min\left(\frac{2-\gamma}{\gamma-1},1\right)$. 

$\bullet$ Both functions are axisymmetric and even in $x_3$.  

$\bullet$ $\rho\ge0$ has compact support (near $B_R$).  

$\bullet$ $\int \rho\,dx = M(\rho_0)$ (independently of $\omega,\ep$).

$\bullet$ The mapping $(\omega,\ep) \to  (\rho, \psi)$  is continuous from 
$(-\bar\omega,\bar\omega) \times (-\bar\ep,\bar\ep)$ into 
$C^1(\overline{B_{2R}})\times C^2_0(\real^3)$. 

More generally, we permit $p(\rho)$ to be any function that satisfies \eqref{cond: p1} and \eqref{cond: p2} below   
and we assume that $M'(\rho_0(0))\ne0$, where $M(\rho(0)) = \int_{\real^3} \rho~dx$ is the total mass of the unique radial nonrotating star solution with center density $\rho(0)$ (more details explained in Theorem 2.1 in \cite{SW2017}). 
\begin{align}\label{cond: p1}
p(s)\in C^3(0,\infty), p'>0, p(0)=0
\end{align}
\begin{align}\label{cond: p2}
&\exists \gamma\in (1,2), \lim_{s\to 0^+}s^{3-\gamma}p'''(s)<0, \text{ and }\\
&\exists \gamma^*\in (\frac65,2), \lim_{s\to\infty} s^{1-\gamma^*}p'(s)>0.\notag
\end{align}

\end{theorem}

\bigskip 
Our construction shows that the solutions are modified from the Lane--Emden solution by a simple 
radial stretching or contraction. The support of $\rho_{\omega,\epsilon}$ takes an oblate shape, 
as we remark at the end of Section 2. The shape is only affected by the magnetic field at higher 
orders in $\omega^2$ and $\epsilon$. 
We remark that the case $\gamma=\frac43$ is excluded from Theorem 
1.1, because in that case the key linearized
operator in our construction has a non-trivial kernel. This corresponds 
to the fact that there exists a
family of {\it non-rotating} radial solutions with zero magnetic field, 
obtained by simple rescaling of an
unperturbed one. The solutions in this family all have the same total 
mass, due to the special scaling
symmetry in this case. With the mass constraint in Theorem 1.1, 
the nearby solutions must come from
this trivial class. However, if we were to remove the mass constraint, 
non-trivial solutions also could arise at the
$\frac43$ power.

There have been many studies, 
including by the giants MacLaurin, Jacobi, Poincar\'e, Liapunov and Chandrasekhar, 
of stationary and steadily rotating stars subject to gravitational forces but without any magnetic field.  
There are two modern methods of analysis of rotating stars, the variational method introduced by 
Auchmuty and Beals \cite{auchmuty1971variational} and Li \cite{li1991uniformly} and the perturbation method introduced by 
Lichtenstein \cite{lichtenstein1933untersuchungen}.  
The perturbation method was recently revived and further developed 
in \cite {SW2017} and \cite {jang2016slowly}, where further references and discussions may be found.   
Furthermore, the two papers \cite{JM2} and \cite{SW2018} appeared after this paper 
was originally submitted. 
There are a number of excellent general expositions, notably the treatises 
\cite{Chandrasekhar}, \cite{Tassoul} and \cite{jardetzky2013theories}.  

However, the only mathematical reference of which we are aware that deals with a magnetic effect is \cite{FLS}, 
which considers a stationary $(v=0)$ magnetic star.  The authors of \cite{FLS} find solutions by a variational method 
and permit $\gamma\ge2$, but they require $k$ to be a constant function of $\psi$. 
Our paper is very different from \cite{FLS} with regard to its methodology and most importantly 
with regard to the rotation of the star.  
Besides permitting rotation, we use a perturbation method and we permit the magnetic current function 
$k(\psi)$ to be completely arbitrary rather than a constant.  
To the best of our knowledge, ours is the first mathematically rigorous result that establishes the existence of rotating magnetic stars.   

In Section 2 we state the assumptions in detail, specialize the model to our situation, and outline the 
proof of the theorem.  Section 3 is devoted to studying the detailed properties of the inverse operator $L^{-1}$.  
Section 4 is devoted to the proof of Fr\'echet differentiability, which is a key requirement of the implicit 
function theorem.  

\section{Setup and outline}

We are looking for axisymmetric steady rotating solutions to the magnetic star equations \eqref{masscons}-\eqref{eq: 5}.  
To that end, let $v=r\omega e_\theta$, $B=B^re_r + B^3e_3$, and assume that all the functions $\rho, B, U$ 
depend only on the cylindrical coordinates $r$ and $x_3$. 
The magnetic star equations simplify considerably under the aforementioned assumptions. 
First of all, by these assumptions we have $\nabla\cdot v=0$, $\nabla \rho\perp e_\theta$, and $v\parallel e_\theta$. As a consequence, the mass conservation equation \eqref{masscons} is automatically satisfied. Equation \eqref{eq: 3} is also satisfied as is seen from the following calculation:
\eqn
\nabla\times(v\times B)=-\nabla_vB+\nabla_Bv=-\omega B^re_\theta+\omega B^re_\theta=0. \eeqn 
Because there is no $\theta$-dependence, equation \eqref{eq: 4} gives us 
\eqn\label{eq: div free}
0  =  r\nb\cdot B  =  \partial_r(rB^r)+ \partial_3(rB^3), \eeqn 
which is 
satisfied if the components of the magnetic field are induced by a scalar axisymmetric 
`magnetic potential' $\psi$ in the following way: 
\eqn 
rB^r=\pa_3\psi, \; rB^3=-\pa_r\psi .  
\eeqn  
This is equivalent to assuming that the vector magnetic potential $A$, given by $B=\nabla\times A$, 
is  $A = -\frac{\psi}r e_\theta$.   

Our next step is to express the term $(\nabla\times B)\times B$ in \eqref{eq: 2} by $\psi$. 
A short computation shows that 
\eqn
\nabla\times B = (\partial_3B^r-\partial_rB^3)e_\theta = r(L\psi)e_\theta,\eeqn
where
\eqn
L\psi =\frac1r\partial_r \left(\frac{\partial_r\psi}{r}\right)+\frac{\partial_3^2\psi}{r^2} 
= \nabla\cdot\left(\frac1{r^2}\nabla \psi\right).\eeqn
So 
\eqn
(\nabla\times B)\times B = (L\psi)e_\theta\times(\pa_3\psi e_r-\pa_r\psi e_3)=-(L\psi)\nabla \psi.\eeqn
Because  $v\cdot \nabla v = -\nabla (\frac12 \omega^2 r^2)$ 
and $\frac{\nabla p}{\rho}=\nabla(h(\rho))$ from \eqref{enthalpy}, 
the momentum equation \eqref{eq: 2} becomes
\eqn\label{eq: 19}
\nabla (-\tfrac12\omega^2 r^2 + h(\rho)) = \nabla U - \tfrac1\rho (L\psi)\nabla \psi  \eeqn
Notice that every term but the last one in \eqref{eq: 19} is a gradient. 
It follows that the last term must be curl free, namely,  
\eqn\label{eq: curl free}
\nabla \left(\frac{L\psi}{\rho}\right)\times \nabla \psi =0.\eeqn 
Thus the gradients of $\frac{L\psi}{\rho}$ and $\psi$ are parallel. A natural 
sufficient condition for this is
that $\frac{L\psi}{\rho}$ is a function of $\psi$. Motivated by this condition, 
we look for a special but quite wide class of
solutions to \eqref{eq: 19} for which $L\psi=\epsilon\rho k(\psi)$ with an 
arbitrarily prescribed function $k$. The
constant $\epsilon$ is conveniently included here as a small parameter. 
For these solutions, the momentum equation \eqref{eq: 19} can now be written as
\eqn
h(\rho)-\tfrac12\omega^2r^2-U + \ep K(\psi)=\text{ constant}, 
\eeqn
where $K(s)=\int_0^s k(t)~dt$. 

To summarize, the magnetic star equations have now been simplified to the following problem:
\begin{align}
h(\rho)-\frac12\omega^2r^2-U + \ep K(\psi)&= \text{ constant\quad in the region  } \{\rho>0\},  \label{eq: new system 1} \\
L\psi &= \ep\rho k(\psi) \quad \text{ in }\real^3,\label{eq: new system 2}\\
U &= \rho*\frac1{|x|} \quad \text{ in }\real^3.\label{eq: new system 3}
\end{align}
The last equation comes from \eqref{eq: 5}, 
together with the assumption that $\rho(x)$ vanishes appropriately at infinity.  
Thus we have \eqref{7}, \eqref{8}, \eqref{9}. 
We also assume the boundary condition $\psi(\infty)=0$. Solutions of \eqref{eq: new system 1}-\eqref{eq: new system 3} together with the boundary condition \eqref{cond: vac bdry} satisfy 
our original system \eqref{masscons}-\eqref{cond: vac bdry}.  The rest of the paper is devoted to the 
existence of these solutions, thereby proving Theorem 1.1.

We will construct solutions to \eqref{eq: new system 1}, \eqref{eq: new system 2}, \eqref{eq: new system 3} 
which are close to the nonrotating, magnetic-free Lane--Emden solutions. 
We thus begin by considering a Lane--Emden solution $\rho_0$ supported on $B_1$, 
as is explained in \cite{SW2017}, and the deformation 
\eqn
\gz(x) = x\left(1+\frac{\zeta(x)}{|x|^2}\right) 
\eeqn  
used in \cite{SW2017}. 
Here $\zeta: B_1\to \real$ is an axisymmetric function that is even in $x_3$.  
If $\zeta$ is small in a suitable norm, $\gz$ is invertible, $\zeta$ can be 
extended to $\real^3$ preserving the symmetry requirements.
The deformation $\gz$ can then be extended to a homeomorphism on $\real^3$ 
(as well as diffeomorphic on $\real^3-\{0\}$) accordingly. See \cite{SW2017} for detailed 
estimates of these facts. We look for a solution of the form
\eqn\label{eq: 45}
\rho_\zeta(z) = \M(\zeta)\rho_0(\gz^{-1}(z)),  \eeqn   
where $\M(\zeta)$ is chosen such that $\rho_\zeta(x)$ has the same mass as $\rho_0(x)$.

Our model \eqref{eq: new system 1}, \eqref{eq: new system 2}, \eqref{eq: new system 3} 
may thus be recast as the pair of equations 
\begin{align}
&\left(\rho_\zeta*\frac{1}{|\cdot|}\right)(z) - \left(\rho_\zeta*\frac{1}{|\cdot|}\right)(0)  
+ \tfrac12\omega^2 (z_1^2+z_2^2) \notag\\
&\quad-h(\rho_\zeta(z))+h(\rho_\zeta(0))- \ep K(\psi(z))+\ep K(\psi(0))=0  
~\text{ for }z\in \gz(\overline{B_1}),\label{eq: 46}\\
&\psi(z) -  \ep L^{-1}(\rho_\zeta k(\psi))(z)=0 ~\text{ for all }z\in \real^3. \label{eq: 47}
\end{align}
The precise definition and properties of $L^{-1}$ are given in Section \ref{sec: background}. 

We reduce the problem further by observing that \eqref{eq: 47} only needs to be solved 
for $z\in \gz(\overline{B_1})$.     Indeed, as $\rho_\zeta$ is supported on $\gz(\overline{B_1})$, 
if we can find a smooth enough function $\psi : \gz(\overline{B_1})\to\real $ 
for which \eqref{eq: 47} holds for all $z\in \gz(\overline{B_1})$, 
then we can extend $\rho_\zeta k(\psi)$ to $\real^3$ by setting it to be zero outside $\gz(\overline{B_1})$.   
Now we extend $\psi$ to $\real^3$ by \eqref{eq: 47} and observe that \eqref{eq: 47} 
holds for all $z\in \real^3$.   In summary, when solving \eqref{eq: 46} and \eqref{eq: 47}, 
we only need them to hold for $z\in \gz(\overline{B_1})$. 

In order that the functions are defined on a fixed domain, 
we make the change of variable  $z=\gz(x)$, and and replace $\psi$ by the function
\eqn
\varphi(x)=\psi(\gz(x)).    \eeqn
It follows that \eqref{eq: 46}, \eqref{eq: 47}   can be written as 
\begin{align}
&\left(\rho_\zeta*\frac{1}{|\cdot|}\right)(\gz(x)) - \left(\rho_\zeta*\frac{1}{|\cdot|}\right)(0)  
+\tfrac12\omega^2 (x_1^2+x_2^2)\left(1+\frac{\zeta(x)}{|x|^2}\right)^2 \notag\\
&\qquad-h(\M(\zeta)\rho_0(x))+h(\M(\zeta)\rho_0(0))- \ep K(\varphi(x))+\ep K(\varphi(0))=0, \\
&\varphi(x) -  L^{-1}(\rho_\zeta \ep k(\varphi(\gz^{-1})))(\gz(x))=0,  
\end{align}
where both equations are now required to be valid only in $\overline{B_1}$. 

Now we begin to set up the scenario for the implicit function theorem.  We define the operators 
\begin{align}
\F_1(\zeta,\varphi,\omega,\epsilon)(x) &= \left(\rho_\zeta*\frac{1}{|\cdot|}\right)(\gz(x)) - \left(\rho_\zeta*\frac{1}{|\cdot|}\right)(0)\notag\\
&\qquad+\tfrac12\omega^2 (x_1^2+x_2^2)\left(1+\frac{\zeta(x)}{|x|^2}\right)^2 \notag\\
&\qquad-h(\M(\zeta)\rho_0(x))+h(\M(\zeta)\rho_0(0))- \epsilon K(\varphi(x))+\epsilon K(\varphi(0)),\\
&\F_2(\zeta,\varphi,\omega,\epsilon)(x) =\varphi(x) - \epsilon L^{-1}(\rho_\zeta k(\varphi(\gz^{-1})))(\gz(x))  
\end{align} 
for $x\in\overline B_1$.    We define $\F=(\F_1,\F_2)^t$. 
Note that 
$\F(0,0,0,0)=0$ because $\rho_0$ is the Lane--Emden solution in $B_1$.  

Let $\dot{B_1}=\overline{B_1}\setminus \{0\}$.
We consider  the space 
\eqn
X=C^1(\overline{B_1})\cap \{f~|~f \text{ is axisymmetric and even in } x_3, 
f(0)=0, \|f\|_X<\infty\}
\eeqn
where the norm is $$\|f\|_{X} = \sup_{x\in \dot{B_1}} \frac{|\nabla f(x)|}{|x|}.$$  
It is easy to see that $X$ is a Banach space. In fact, a Cauchy sequence $\{u_n\}$ in $X$ is also Cauchy and thus converges to $u$ in $C^1(\overline{ B_1})$. Now for any $x\in \dot{B_1}$, $|\nabla u_n(x)-\nabla u_m(x)| < \epsilon |x|$ for $n, m$ sufficiently large. Taking $m$ to infinity, we get $|\nabla u_n(x) - \nabla u(x)| < \epsilon |x|$, which means $\{u_n\}$ converges to $u$ in $X$.
Finally, we define $X_\delta  =  \{f\in X\ | \ \|f\|_X \le \delta\}$. 

\begin{lemma}   \label{Frechet}
The operator $\F$ is continuously Fr\'{e}chet differentiable from 
$X_\delta\times X_\delta \times \real \times\real$ into $X\times X$ provided $\delta$ is sufficiently small.
\end{lemma}  
Lemma \ref{Frechet} will be proven in Section 4.  
\begin{lemma}  \label{kernel}
$\frac{\partial \F}{\partial (\zeta,\varphi)}(0,0,0,0):X^2\to X^2$ is an isomorphism. 
\end{lemma}
\begin{proof}  We write $\frac{\partial \F}{\partial (\zeta,\varphi)}(0,0,0,0): X^2\to X^2$ as a matrix:
\eqn
\begin{pmatrix}\frac{\partial \F_1}{\partial \zeta} &  \frac{\partial \F_1}{\partial \varphi}\\ \frac{\partial \F_2}{\partial \zeta}  & \frac{\partial \F_2}{\partial \varphi} \end{pmatrix}(0,0,0,0) = \begin{pmatrix} \frac{\partial \F_1}{\partial \zeta}(0,0,0,0)& 0 \\ 0 & I \end{pmatrix}.
\eeqn
By Theorem 4.1 in \cite{SW2017}, and the fact that $\L$ in \cite{SW2017} is the same as $\frac{\partial \F_1}{\partial \zeta} (0,0,0,0)$, we immediately see that $\frac{\partial \F}{\partial (\zeta,\varphi)}(0,0,0,0)$ is an isomorphism.
\end{proof}

Given these two lemmas and the fact that $\F(0,0,0,0)=0$,  
the implicit function theorem provides a solution for every small enough $\epsilon$ and $\omega$. 

In addition, with the linearized operators, we may study the first order approximation of $\zeta(\omega,\epsilon)$ and $\varphi(\omega,\epsilon)$ as functions of $\omega^2$ and $\epsilon$. In fact, one easily obtains
\eqn
\begin{pmatrix} \zeta(\omega,\epsilon) \\ \varphi(\omega,\epsilon)\end{pmatrix} = \begin{pmatrix}\frac{\partial\zeta}{\partial \omega^2}(0,0)\omega^2 + \frac{\partial\zeta}{\partial\epsilon}(0,0)\epsilon\\ \frac{\partial\varphi}{\partial \omega^2}(0,0)\omega^2 + \frac{\partial\varphi}{\partial\epsilon}(0,0)\epsilon \end{pmatrix} + o(|\omega|^2+|\epsilon|),
\eeqn
where 
\begin{align*}
\begin{pmatrix}\frac{\partial \zeta}{\partial \omega^2}(0,0)\\\frac{\partial\varphi}{\partial \omega^2}(0,0) \end{pmatrix} &= -\begin{pmatrix} \frac{\partial \F_1}{\partial \zeta}(0,0,0,0) & 0 \\ 0 & I\end{pmatrix}^{-1}\begin{pmatrix}\frac12 r^2\\0 \end{pmatrix}\\
&= \begin{pmatrix} -\frac12 [\frac{\partial \F_1}{\partial \zeta}(0,0,0,0)]^{-1} r^2\\ 0 \end{pmatrix},
\end{align*}
and
\begin{align*}
\begin{pmatrix}\frac{\partial \zeta}{\partial \epsilon}(0,0)\\\frac{\partial\varphi}{\partial \epsilon}(0,0) \end{pmatrix} &= -\begin{pmatrix} \frac{\partial \F_1}{\partial \zeta}(0,0,0,0) & 0 \\ 0 & I\end{pmatrix}^{-1}\begin{pmatrix}0\\ -L^{-1}(\rho_0k(0)) \end{pmatrix}\\
&= \begin{pmatrix} 0\\  k(0)L^{-1}(\rho_0)\end{pmatrix}.
\end{align*}
In other words, the first order approximation of $\zeta(\omega,\epsilon)$ and $\varphi(\omega,\epsilon)$ are given by $-\frac{\omega^2}2 [\frac{\partial \F_1}{\partial \zeta}(0,0,0,0)]^{-1} r^2$ and $\epsilon k(0) L^{-1}(\rho_0)$ respectively.

Notice that the magnetic field does not affect the lowest order change in the shape of the star, because $\epsilon$ does not appear in the approximation for $\zeta(\omega,\epsilon)$. By the discussion in Section 4.6 of \cite{SW2017}, we see that the first order approximation of $\zeta(\omega,\epsilon)$ gives rise to an oblate star, which is wider at the equator than at the poles. On the other hand, if $k(0)\ne 0$, the first order approximation of $\varphi(\omega,\epsilon)$ is necessarily nonzero, giving rise to a genuine magnetic star. 

\section{The operator $L^{-1}$}\label{sec: background}

Before we prove the Fr\'echet differentiability of $\F$, we carefully define $L^{-1}$ and 
compile a few useful estimates of it.   Consider the equation $Lu=f$ in $\real^3$, 
where both $u$ and $f$ are axisymmetric and even in $x_3$, where $f$ is compactly supported, 
and where $u(\infty)=0$.   If we define $v={u}/{r^2}$, then a simple calculation gives
\eqn\label{eq: 55}
v_{rr}+\frac3rv_r + v_{zz}= f.   \eeqn 
We write $z=x_3,\ r=\sqrt{x_1^2+x_2^2},\ \hat v(r,z) = v(x_1,x_2,x_3)$.  
Following \cite{FLS}, we use the five-dimensional extension of the equation to obtain a simple explicit formula for the solution. Let $\tilde{v}$ be the 5D extensions of $v$ defined by 
\begin{align}
\tilde{v}(x_1,x_2,x_3,x_4,x_5) &= v\left(\sqrt{x_1^2+x_2^2+x_3^2+x_4^2}\cos \theta, \sqrt{x_1^2+x_2^2+x_3^2+x_4^2}\sin \theta, x_5\right)\notag\\
&=  \hat v\left(\sqrt{x_1^2+x_2^2+x_3^2+x_4^2}, x_5\right),
\end{align}
with $\tilde{f}$  defined in the same way.
Thus \eqref{eq: 55} can be written as 
$\Delta_5\tilde v = \tilde f.$  
Moreover, the condition $u(\infty)=0$ implies $\tilde{v}(\infty) =0$. 
It follows that 
\begin{align}
v(x_1,x_2,x_3) &= v(\sqrt{x_1^2+x_2^2},0,x_3) = \tilde v(x_1,x_2,0,0,x_3)\notag\\
&= C_5\int_{\real^5}\frac{1}{|(x_1,x_2,0,0,x_3)-y|^3}\tilde f(y)~dy.  
\end{align}
Thus 
\begin{equation}\label{eq: 60}
(L^{-1}f)(x)=u(x_1,x_2,x_3)= C_5(x_1^2+x_2^2)(L_1^{-1}f)(x),  
\end{equation}
where we define the integral operator
\eqn
(L_1^{-1}f)(x) = \int_{\real^5}\frac{1}{|(x_1,x_2,0,0,x_3)-y|^3}\tilde f(y)~dy.  \eeqn
Since 
$$\left|\sqrt{x_1^2+x_2^2+x_3^2+x_4^2}-\sqrt{y_1^2+y_2^2+y_3^2+y_4^2}\right|\le |x-y|$$
for $x,y\in\real^5$, there exists a constant $C=C(\beta,R)$ such that 
\eqn
\|\tilde f\|_{C^{0,\beta}(B_R(\real^5))}\le C\| f\|_{C^{0,\beta}(B_R(\real^3))}  \eeqn
for any $0<\beta\le 1$ and $R>0$. In other words, $\tilde f$ is as regular as $f$. 
It is worthwhile to keep in mind that, by \eqref{eq: 60}, $L^{-1}$ is basically a quadratic function vanishing at the origin multiplied by the inverse Laplacian in 5D. Therefore, in addition to enjoying all the regularization properties of the inverse Laplacian, it also automatically vanishes to quadratic degree at the origin.   This observation gives rise to the next lemma.

\begin{lemma}\label{lem: 1.1} Let $f\in L^\infty(B_2)$. Then $L^{-1}f \in C^{1,\beta}(\overline{B_2})$ for every $0<\beta<1$. There is a constant $C>0$ such that   
\eqn
| L^{-1}f(z) | + | \nabla L^{-1}f (z)| \leq C \|f\|_{L^\infty(B_2)}|z|
\eeqn
for any $z\in B_2$. 
\end{lemma}
\begin{proof} 
Using \eqref{eq: 60} and its gradient, we have 
\begin{align}\label{eq: grad L inv}
(\nabla L^{-1} f)(z) &= C_5\nabla(z_1^2+z_2^2)(L_1^{-1}f)(z)+C_5(z_1^2+z_2^2)(\nabla L_1^{-1}f)(z),
\end{align}
the inequality follows from the boundedness 
of $\int_{B_2\subset \real^5} \frac{1}{|x-y|^3} dy$ and  $\int_{B_2\subset \real^5} \frac{1}{|x-y|^4} dy$. 
\end{proof}

Combining the $C^{1,\beta}$ regularity of $L_1^{-1}f$ with the properties of $\gz$, we get the next lemma. 

\begin{lemma}\label{lem: 1.2} Let $f\in L^\infty(B_2)$, $\zeta_1,\zeta_2\in X_\delta$ where $\delta>0$ is sufficiently small. For any $0<\beta<1$, there exists $C_\beta>0$ such that 
\eqn
| [\nabla L^{-1}f ]( \gzone(x)) -  [\nabla L^{-1}f] ( \gztwo(x)) | \leq C_\beta \|f\|_{L^\infty(B_2)}\|\zeta_1-\zeta_2\|_X^\beta|x|
\eeqn
for $x\in \overline{B_1}$. 
\end{lemma}

\begin{proof}
By \eqref{eq: grad L inv}, $\nabla L^{-1}f$ is of the form $f_1+(z_1^2+z_2^2)f_2$, where $\|f_1\|_{C_1(B_2)}\le C\|f\|_{L^\infty}$, and $\|f_2\|_{C^{0,\beta}(B_2)}\le C_\beta \|f\|_{L^\infty}$. It follows that
\begin{align}
&~| [\nabla L^{-1}f ]( \gzone(x)) -  [\nabla L^{-1}f] ( \gztwo(x)) | \notag\\
\le &~C\|f\|_{L^\infty}|\gzone(x)-\gztwo(x)|+C_\beta\|f\|_{L^\infty} |\gzone(x)||\gzone(x)-\gztwo(x)|^\beta\notag\\
\le &~C_\beta \|f\|_{L^\infty(B_2)}\|\zeta_1-\zeta_2\|_X^\beta|x|.
\end{align}
In the last step, we used $|\gzone(x)-\gztwo(x)|\le \|\zeta_1-\zeta_2\|_X|x|$, and $|\gzone(x)|\le (1+\|\zeta\|_X)|x|$ (cf. Lemma 3.4 of \cite{SW2017}).
\end{proof}

The standard elliptic Schauder estimates immediately imply the following lemma. 
\begin{lemma}\label{lem: D2 L inv}
Let $f\in C^{0,\beta}(\overline{B_2})$ for some $0<\beta<1$. Then $L^{-1}f\in C^{2,\beta}(\overline{B_2})$, and there is a constant $C_\beta>0$ such that
\eqn
\| \nabla^2 L^{-1}f \|_{C^{0,\beta}(\overline{B_2})} \leq C_\beta \|f\|_{C^{0,\beta}(B_2)}.\eeqn
\end{lemma}

Next we discuss a few more delicate estimates involving interior composition with the deformation $\gz^{-1}$ before the action of $L^{-1}$. To that end, we first write an alternative expression for $L^{-1}(u(\gz^{-1}))(z)$: 
\begin{align}
L^{-1}(u(\gz^{-1}))(z)&= C_5(z_1^2+z_2^2)\int_{\real^5}\frac{1}{|(z_1,z_2,0,0,z_3)-y|^3}\widetilde{u(\gz^{-1})}(y)~dy\notag\\
&=C_5(z_1^2+z_2^2)\int_{\real^5}\frac{1}{|(z_1,z_2,0,0,z_3)-y|^3}\tilde{u}(\widetilde{\gz}^{-1})(y)~dy\notag\\
&=C_5(z_1^2+z_2^2)\int_{\text{supp }\tilde u}\frac{1}{|(z_1,z_2,0,0,z_3)-\widetilde{\gz}(y)|^3}\tilde{u}(y)\det D \widetilde{\gz}(y)~dy\label{eq: 290}
\end{align} 
where we define $\widetilde{\gz}:\real^5\to \real^5$ as 
\eqn
\widetilde{\gz}(x) = x\left(1+\frac{\tilde{\zeta}(x)}{|x|^2}\right).   \eeqn
In the preceding identity we have used the nontrivial but straightforward fact that 
\eqn
\widetilde{u(\gz^{-1})}(y)=\tilde{u}(\widetilde{\gz}^{-1})(y).
\eeqn
Similarly, for a row-vector-valued function $u_a$ and a column-vector-valued function $u_b$, 
we obtain the alternative expression 
\begin{align}\label{eq: L^-1 alt 2}
&~L^{-1}\left(u_a(\gz^{-1})D (\gz^{-1})u_b(\gz^{-1})\right)(z)\notag\\
=&~C_5(z_1^2+z_2^2)\int_{\real^5}\frac{[u_a(\nabla\gz)^{-1}u_b]^{~\widetilde{ }}(y)}{|(z_1,z_2,0,0,z_3)-\widetilde{\gz}(y)|^3}\det D \widetilde{\gz}(y)~dy.
\end{align}

The reason we write such alternative expressions is that we want to estimate the difference 
$ L^{-1}(u(\gzone^{-1})) -L^{-1}(u(\gztwo^{-1}))$ assuming only an $L^\infty$ control on $u$.   
In order to get an estimate depending on $\|\zeta_1-\zeta_2\|_X$, 
we must use \eqref{eq: 290} to move the interior composition of $\gz^{-1}$ out of $ u$. 
The following lemma will be useful in deriving the difference estimate of 
$ L^{-1}(u(\gzone^{-1})) -L^{-1}(u(\gztwo^{-1}))$: 

\begin{lemma}\label{lem: 2.3} Suppose $\zeta_1,\zeta_2\in X_\delta$ where $\delta>0$ 
is sufficiently small and $u\in L^\infty(B_1)$. Then
\begin{align}\label{eq: lem 1.4}
&~\left|\nabla_z \int_{B_1\subset \real^5}\left(\frac{1}{|(z_1,z_2,0,0,z_3)-\widetilde{\gzone}(y)|^3}-\frac{1}{|(z_1,z_2,0,0,z_3)-\widetilde{\gztwo}(y)|^3}\right)\tilde u(y)~dy\right| \notag\\
\le &~Cd(1+|\log d|)\|u\|_{L^{\infty}},
\end{align}
where $d=\|\zeta_1-\zeta_2\|_X$, for $z\in B_2\subset\real^3$.
\end{lemma}
\begin{proof}
The singularities of the integral are at $y_{(1)}=\widetilde{\gzone}^{-1}(z_1,z_2,0,0,z_3)$ 
and $y_{(2)}=\widetilde{\gztwo}^{-1}(z_1,z_2,0,0,z_3)$.   
For any $y\in \overline{B_2}$, we claim that $|(z_1,z_2,0,0,z_3)-\widetilde{\gzone}(y)|$ 
is comparable to $|y_{(1)}-y|$, and that $|(z_1,z_2,0,0,z_3)-\widetilde{\gztwo}(y)|$ 
is comparable to $|y_{(2)}-y|$. In fact,  
$$
|(z_1,z_2,0,0,z_3)-\widetilde{g_{\zeta_i}}(y)| = |\widetilde{g_{\zeta_i}}(y_{(i)})-\widetilde{g_{\zeta_i}}(y)|\le \|D\widetilde{g_{\zeta_i}}\|_{L^\infty}|y_{(i)}-y|.
$$
Similarly
$$
|y_{(i)}-y|\le \|D\widetilde{g_{\zeta_i}}^{-1}\|_{L^\infty}|(z_1,z_2,0,0,z_3)-\widetilde{\gz}(y)|.
$$
The $L^\infty$ bound on $D\widetilde{g_{\zeta_i}}$ and $D\widetilde{g_{\zeta_i}}^{-1}$ 
follows in a similar fashion to Lemma 3.4 of \cite{SW2017}.   
Using similar estimates involving differences of $\zeta_1$ and $\zeta_2$, 
we obtain that $|\widetilde{\gzone}(y)-\widetilde{\gztwo}(y)|$ and $|y_{(1)}-y_{(2)}|$ 
are both bounded by a constant multiple of $\|\zeta_1-\zeta_2\|_X$.
By the preceding distance estimates, we may choose a ball $B_d$ to be centered at the midpoint 
of $y_{(1)}$ and $y_{(2)}$ with radius comparable to $d=\|\zeta_1-\zeta_2\|_X$, 
such that the following facts hold whenever $y$ is outside $B_d$:
\begin{enumerate}[(a)]
\item $|y-y_{(i)}|\ge 2|y_{(1)}-y_{(2)}|$, $i=1,2$.
\item $|(z_1,z_2,0,0,z_3)-\widetilde{g_{\zeta_i}}(y)|\ge 2|\widetilde{\gzone}(y)-\widetilde{\gztwo}(y)|$, $i=1,2$.
\end{enumerate}
We split the integral into one piece on $B_d$ and another piece off $B_d$. On $B_d$, 
we use the fact that $|(z_1,z_2,0,0,z_3)-\widetilde{g_{\zeta_i}}(y)|$ is comparable to $|y_{(i)}-y|$, 
so the integral is bounded by 
\eqn
C\int_0^d \frac{2}{r^4}r^4~dr \|u\|_{L^{\infty}} = Cd\|u\|_{L^{\infty}}.  \eeqn
Off $B_d$, we use the distance estimates above to conclude that for all $0\le t\le 1$, 
$|(z_1,z_2,0,0,z_3)-t\widetilde{\gzone}(y)-(1-t)\widetilde{\gztwo}(y)|$ is comparable to the distance 
between $y$ and the center of $B_d$, so the integral is bounded by 
\eqn
C\int_d^1\frac{1}{r^5}r^4~dr \|\gzone-\gztwo\|_{\infty}\|u\|_{\infty}\le Cd|\log d|\|u\|_\infty.  \eeqn
\end{proof}

Now we use the preceding lemma to prove the relevant estimate on $L^{-1}$.

\begin{lemma}
\label{lem: 1.5}
Let $u$, $u_a$, $u_b$ be respectively scalar, row-vector-valued and column-vector-valued 
bounded functions supported on $\overline{B_1}$. 
Let $\zeta_1,\zeta_2\in X_\delta$ where $\delta>0$ is sufficiently small. 
Then there is a constant $C>0$ such that
\begin{align}\label{est: lem 1.5.1}
&~|\{\nabla L^{-1}[u(\gzone^{-1})]\}(z)-\{\nabla L^{-1}[u(\gztwo^{-1})]\}(z)|\notag\\
\le &~C d(1+|\log d|)\|u\|_{L^\infty}|z|,
\end{align}
and
\begin{align}\label{est: lem 1.5.2}
~|\{\nabla L^{-1}[u_a(\gzone^{-1})\nabla (\gzone^{-1})u_b(\gzone^{-1})]\}(z)
& - \{\nabla L^{-1}(u_a(\gztwo^{-1})\nabla (\gztwo^{-1})u_b(\gztwo^{-1})]\}(z)|\notag\\
\le &~Cd(1+|\log d|)\|u_1\|_{L^\infty}\|u_2\|_{L^\infty}|z|.
\end{align}
Here $d=\|\zeta_1-\zeta_2\|_X$.
\end{lemma}

\begin{proof}
We write $\{\nabla L^{-1}[u(\gzone^{-1})]\}(z)-\{\nabla L^{-1}[u(\gztwo^{-1})]\}(z)$ as
\begin{align}
&C_5 \nabla(z_1^2+z_2^2)\left(L_1^{-1}[u(\gzone^{-1})](z)-L^{-1}[u(\gztwo^{-1})](z)\right)\notag\\
&\qquad +C_5(z_1^2+z_2^2)\left(\{\nabla L_1^{-1}[u(\gzone^{-1})]\}(z)-\{\nabla L_1^{-1}[u(\gztwo^{-1})]\}(z)\right).
\end{align}
Using \eqref{eq: 290}, $L_1^{-1}[u(\gzone^{-1})](z)-L^{-1}[u(\gztwo^{-1})](z)$ can be written as
\begin{align*}
&~\int_{B_1}\frac{\tilde u(y)D  \widetilde{\gzone}(y)}{|(z_1,z_2,0,0,z_3)-\widetilde{\gzone}(y)|^3}- \frac{\tilde u(y)D  \widetilde{\gztwo}(y)}{|(z_1,z_2,0,0,z_3)-\widetilde{\gztwo}(y)|^3}~dy\notag\\
 = &~\int_{B_1}\frac{\tilde u(y)[D  \widetilde{\gzone}(y)-D  \widetilde{\gztwo}(y)]}{|(z_1,z_2,0,0,z_3)-\widetilde{\gzone}(y)|^3}~dy \notag\\
 &\quad + \int_{B_1}\tilde u(y)D  \widetilde{\gztwo}(y)\left(\frac{1}{|(z_1,z_2,0,0,z_3)-\widetilde{\gzone}(y)|^3}-\frac{1}{|(z_1,z_2,0,0,z_3)-\widetilde{\gztwo}(y)|^3}\right)~dy
\end{align*}
which is easily seen to be bounded by $Cd\|u\|_{L^\infty}$ in view of the estimates $|D  \widetilde{\gzone}(y)-D  \widetilde{\gztwo}(y)|\le Cd$ and $|\widetilde{\gzone}(y)-\widetilde{\gztwo}(y)|\le Cd$. We can treat $\{\nabla L_1^{-1}[u(\gzone^{-1})]\}(z)-\{\nabla L_1^{-1}[u(\gztwo^{-1})]\}(z)$ in a similar way, this time using $|D  \widetilde{\gzone}(y)-D  \widetilde{\gztwo}(y)|\le Cd$ and Lemma \ref{lem: 2.3} to draw the conclusion. The difference expression \eqref{est: lem 1.5.2} can be treated analogously using \eqref{eq: L^-1 alt 2}.
\end{proof}

Finally, we will also need a lemma concerning the H\"{o}lder estimate of $u(\gzone^{-1})-u(\gztwo^{-1})$.  

\begin{lemma}\label{lem: 2.4}
If $u\in C^{0,\beta}(\real^3)$ for some $0<\beta\le 1$, and $u$ is supported in $B_1$, and $\zeta_1,\zeta_2\in X_\delta$ where $\delta>0$ is sufficiently small, then $\|u(\gzone^{-1})-u(\gztwo^{-1})\|_{C^{0,\alpha}(\real^3)}\le C_\alpha\|\zeta_1-\zeta_2\|_X^{\beta-\alpha}$ for every $0<\alpha<\beta$.
\end{lemma}
\begin{remark}
Notice that $\|u(\gzone^{-1})-u(\gztwo^{-1})\|_{C^{0,\alpha}(\real^3)}$ might not tend to zero 
as $\|\zeta_1-\zeta_2\|_X\to 0$ for $\alpha=\beta$, as is suggested by the simple example 
$f_{\epsilon}(x)=|x+\epsilon|^\beta-|x|^\beta$. 
The $C^{0,\beta}$ norm of $f_\epsilon$ does not decrease as $\epsilon\to 0$.
\end{remark}
\begin{proof}
Let $v=u(\gzone^{-1})-u(\gztwo^{-1})$, 
we have $\|v\|_{C^{0,\beta}}\le C$, and $\|v\|_{\infty}\le C\|\zeta_1-\zeta_2\|_X^{\beta}$. In fact, $v(x) = u(g_{\zeta_1}^{-1}(x)) - u(g_{\zeta_2}^{-1}(x))$ is bounded by a constant multiple of $|g_{\zeta_1}^{-1}(x)-g_{\zeta_2}^{-1}(x)|^\beta$, due to the H\"older continuity of $u$. The latter is bounded by $\|\zeta_1-\zeta_2\|_X^\beta$ by the property of $g_\zeta$ given in Lemma 3.4 of \cite{SW2017}.
For any $x,y$, $x\neq y$, note that 
\[
\begin{split}
\frac{|v(x) - v(y)|}{ |x-y|^\alpha} &= \left( \frac{|v(x) - v(y)|}{ |x-y|^\beta} \right)^{\frac{\alpha}{\beta}}  |v(x) - v(y)|^{1 - \frac{\alpha}{\beta}} \\
&\leq C\|v\|_{C^{0,\beta}}^{\frac{\alpha}{\beta}}  \|v\|_\infty^{1-\frac{\alpha}{\beta}}\\
& \leq C 
\|\zeta_1-\zeta_2\|_X^{\beta-\alpha},
\end{split}
\]
from which we deduce the result. 

\end{proof}

\section{Fr\'{e}chet differentiability}

In this section, we prove the Fr\'{e}chet differentiability of $\F(\zeta,\varphi,\omega,\epsilon)$. 

\begin{theorem} \label{Frechet theorem}
The operator $\F: X_\delta^2\times \real^2\to X^2$ is continuously Fr\'echet differentiable 
if $\delta>0$ is sufficiently small .    
\end{theorem}

\subsection{$\F$ maps into $X^2$}

\begin{lemma}\label{lem: 2.1} There exists a constant $C>0$ depending on $\rho_0$, $k$, and $\delta$ such that 
\eqn
\|\F_1(\zeta,\varphi,\omega,\epsilon)\|_X \leq C(1+\omega^2+\epsilon) \;\;\text{and}\;\; \|\F_2(\zeta,\varphi,\epsilon) \|_X \leq C (1+\epsilon)
\eeqn
if $\zeta\in X_\delta$ and $\varphi\in X_\delta$ for sufficiently small $\delta>0$. 
\end{lemma}

\begin{proof} We start with $\F_1$. The terms except $- \epsilon K(\varphi(x))+\epsilon K(\varphi(0))$ in $\F_1$ have been shown to map into $X$ in \cite{SW2017} (cf. Lemma 5.1 in \cite{SW2017}).  In order for $- \epsilon K(\varphi(x))+\epsilon K(\varphi(0))$ to also map into $X$, it suffices to show that $|K'(\varphi(x))\nabla \varphi(x)|$ is bounded by $C|x|$. But this immediately follows from the fact that $\varphi\in X$ and $|K'(\varphi(x))|$ is bounded. 

We next move on to $\F_2$. Let us rewrite $\F_2$ as 
\eqn\label{eq: 12}
\F_2(\zeta,\varphi,\omega,\epsilon) =\varphi(x) - \epsilon \M(\zeta ) \F_3(\zeta,\varphi) 
\eeqn
where 
\eqn\label{eq: 63}
\F_3(\zeta,\varphi) := L^{-1}(\rho_0(  \gz^{-1})\ k(\varphi(\gz^{-1})))(\gz(x)) . 
\eeqn
Since $\varphi \in X$, we only need to show that $\F_3$ maps into $X$. To this end, we compute 
the spatial derivative of \eqref{eq: 63}:
\eqn
\partial_i \F_3(\zeta,\varphi) = \left[\nabla L^{-1}(\rho_0(  \gz^{-1})\ k(\varphi(\gz^{-1})))\right](\gz(x))\cdot \partial_i\gz(x).
\eeqn
Since $|\gz(x)-x|\leq \|\zeta\|_X |x|$ and $|\partial_i \gz (x)|\leq C(1+\|\zeta\|_X)$ (cf. Lemma 3.4 in \cite{SW2017}), in order to bound $|\partial_i \F_3|$ by $C|x|$, it is sufficient to show that $\left[\nabla L^{-1}(\rho_0(  \gz^{-1}) k(\varphi(\gz^{-1})))\right](z)$ is bounded by $C|z|$. This, in turn, is a consequence of Lemma \ref{lem: 1.1}.

\end{proof}

\subsection{Formal derivative of $\F$}

To simplify notation, we suppress the $\omega$, $\epsilon$ dependence in $\F$.  
Another reason for doing so is that the differentiability with respect to $\omega$ and $\epsilon$ is simpler. 
Therefore, for the moment, we think of them as being fixed. 
Let $\zeta$, $\varphi \in X_\delta$ and $\xi,$ $\eta\in X$ be given. 
Consider $s\in \real$ in a sufficiently small neighborhood of $0$ so that 
$\zeta+s\xi$, $\varphi+s\eta\in X_\delta$. We define the {\it formal derivatives} of $\F$ 
with respect to $\zeta$, and $\varphi$  respectively as the pointwise limits  
\begin{equation}\label{def: formal deriv}
\begin{split}
&\left[\frac{\pa \F}{\pa \zeta}(\zeta,\varphi)\xi \right](x)  = \partial_s\bigg|_{s=0}\F(\zeta+s\xi,\varphi)(x),\\
& \left[\frac{\pa \F}{\pa \varphi}(\zeta,\varphi)\eta \right](x)  = \partial_s\bigg|_{s=0}\F(\zeta,\varphi+s\eta)(x), 
\end{split}
\end{equation}
for every {\it fixed} $x$. We do not yet claim that this is a Gateaux derivative, which would require 
a specific condition making use of the norm, while for the time being 
we are only defining it pointwise.
The formal derivative of $\F_1$ with respect to $\zeta$ was computed in Lemma 5.2 in \cite{SW2017}.  
It is easy to see that the formal derivative of $\F_1$ with respect to $\varphi$ is   
\eqn\label{eq: 17}
 \left[\frac{\pa \F_1}{\pa \varphi}(\zeta,\varphi)\eta \right](x) = - \epsilon K'(\varphi(x)) \eta(x) 
\eeqn

As for $\F_2$ it is clear from \eqref{eq: 12} that
\eqn
 \left[\frac{\pa \F_2}{\pa \zeta}(\zeta,\varphi)\xi \right] = - \epsilon [\M'(\zeta)\xi] \F_3(\zeta,\varphi)  - \epsilon \M(\zeta) 
 \left[\frac{\partial \F_3}{\partial \zeta}(\zeta,\varphi)\xi \right]
\eeqn
and 
\eqn
 \left[\frac{\pa \F_2}{\pa \varphi}(\zeta,\varphi)\eta \right] = \eta   - \epsilon \M(\zeta) 
 \left[\frac{\partial \F_3}{\partial \varphi}(\zeta,\varphi)\eta \right]
\eeqn
where $\F_3$ is given in  \eqref{eq: 63}.  
According to (5.18) in \cite{SW2017}, $\M'(\zeta) \xi$ is given by  
\eqn   \label{eq: M'}
\begin{split}
\M'(\zeta) \xi&=\frac{-M}{\left(\int_{B_1}\rho_0(x) \det Dg_{\zeta}(x)~dx\right)^2} \times  \\
&\quad  \times \int_{B_1}\rho_0(x)\det Dg_{\zeta}(x)   
\tr \left[(Dg_{\zeta})^{-1}(x) D\left(\xi(x)\frac{x}{|x|^2}\right)\right]~dx 
\end{split}
\eeqn 

The computation of $\left[\frac{\partial \F_3}{\partial \zeta}(\zeta,\varphi)\xi \right]$ is similar to the one 
for the first term in $\F_1$, both of which are integral operators involving $\rho_0(\gz^{-1})$. 
Using the formula (see (5.30) of \cite{SW2017}),
\eqn
\partial_s g_{\zeta+s\xi}^{-1}(y) = -Dg_{\zeta+s\xi}^{-1}(y)\xi(g_{\zeta+s\xi}^{-1}(y))\frac{g_{\zeta+s\xi}^{-1}(y)}{|g_{\zeta+s\xi}^{-1}(y)|^2},
\eeqn
we deduce that 
\begin{align}
&~\left[\frac{\partial \F_3}{\partial \zeta}(\zeta,\varphi)\xi \right](x) \notag\\
=&~-L^{-1}\bigg[\left(\nabla\rho_0(\gz^{-1})k(\varphi(\gz^{-1}))+\rho_0(\gz^{-1})k'(\varphi(\gz^{-1})) \nabla \varphi(\gz^{-1})\right)\cdot\notag\\
&~\qquad\qquad D( \gz^{-1})\cdot\frac{\xi(\gz^{-1})\gz^{-1}}{|\gz^{-1}|^2}\bigg](\gz(x))\notag\\
&~+\left[\nabla L^{-1}(\rho_0(\gz^{-1}) k(\varphi(\gz^{-1})))\right](\gz(x))\cdot \frac{\xi(x)}{|x|^2}x.
\end{align}
Lastly, the formal derivative with respect to $\varphi$ is given by 
\eqn
\left[\frac{\partial \F_3}{\partial \varphi}(\zeta,\varphi)\eta\right](x) = L^{-1} \left[ \rho_0(\gz^{-1})k'(\varphi(\gz^{-1})) \eta (\gz^{-1}) \right] (\gz(x)). 
\eeqn
A rigorous justification of the above formulas involving the cut-off function method 
may be found in Lemma 5.2. of \cite{SW2017}. The details are omitted.

Next we will show that the formal derivative $\frac{\pa \F}{\pa (\zeta,\varphi)}$, just computed,  
is a  bounded linear map on $X\times X$.

\begin{lemma}\label{lem: 2.2}
If $\zeta$, $\varphi\in X_\delta$ and $\delta$ is sufficiently small, there exists a constant $C>0$ such that 
\eqn\label{eq: 23}
\left\| \frac{\pa \F_1}{\pa \zeta}  (\zeta,\varphi)\xi \right \|_X\leq C (1+\omega^2)\|\xi\|_X, \;\; 
\left\| \frac{\pa \F_1}{\pa \varphi}  (\zeta,\varphi)\eta \right \|_X\leq C \epsilon \|\eta \|_X
\eeqn
and 
\eqn\label{eq: 24}
\left\| \frac{\pa \F_2}{\pa \zeta}  (\zeta,\varphi)\xi \right \|_X\leq C \epsilon \|\xi\|_X, \;\; 
\left\| \frac{\pa \F_2}{\pa \varphi}  (\zeta,\varphi)\eta  \right\|_X\leq C (1+\epsilon )\|\eta \|_X. 
\eeqn
\end{lemma}

\begin{proof} The first inequality in \eqref{eq: 23} was derived in Lemma 5.5 in \cite{SW2017}. The second inequality in \eqref{eq: 23} directly follows from \eqref{eq: 17} since $|K'(\varphi(x))|$ is bounded. The rest of the proof is devoted to \eqref{eq: 24}. It suffices to prove the boundedness of $\frac{\pa \F_3}{\pa \zeta}  (\zeta,\varphi)$ and $\frac{\pa \F_3}{\pa \varphi}  (\zeta,\varphi)$ on $X$. 

We start with $\frac{\pa \F_3}{\pa \zeta}  (\zeta,\varphi)$.  
To estimate its $X$ norm, we take the spatial derivatives  
\begin{align}
&~\partial_i \left[\frac{\partial \F_3}{\partial \zeta}(\zeta,\varphi)\xi \right](x)\notag\\
=&~-\nabla L^{-1}\bigg[\left(\nabla\rho_0(\gz^{-1})k(\varphi(\gz^{-1}))+\rho_0(\gz^{-1})k'(\varphi(\gz^{-1}))\nabla \varphi(\gz^{-1})\right)\cdot\notag\\
&~\qquad\qquad D \gz^{-1}\cdot\frac{\xi(\gz^{-1})\gz^{-1}}{|\gz^{-1}|^2}\bigg](\gz(x))\cdot\partial_i \gz(x)\label{eq: 68}\\
&~+\left[\nabla^2 L^{-1}(\rho_0(\gz^{-1}) k(\varphi(\gz^{-1})))\right](\gz(x)) \partial_i \gz(x)\cdot \frac{\xi(x)}{|x|^2}x\label{eq: 69}\\
&~+\left[\nabla L^{-1}(\rho_0(\gz^{-1}) k(\varphi(\gz^{-1})))\right](\gz(x))\cdot \partial_i\left(\frac{\xi(x)}{|x|^2}x\right).\label{eq: 70}
\end{align}
We claim, starting with \eqref{eq: 68}, that each term in the right-hand side is bounded by $C\|\xi\|_X|x|$. 
Due to Lemma \ref{lem: 1.1} and the estimates $|\gz(x)-x|\leq \|\zeta\|_X |x|$ and 
$|\partial_i \gz (x)|\leq C(1+\|\zeta\|_X)$ (cf. Lemma 3.4 in \cite{SW2017}), we can bound \eqref{eq: 68} by 
\[
\begin{split}
|\eqref{eq: 68}| \leq C\bigg\|\left[ \nabla\rho_0(\gz^{-1})k(\varphi(\gz^{-1}))+\rho_0(\gz^{-1})k'(\varphi(\gz^{-1}))\nabla \varphi(\gz^{-1})\right]  \cdot  & \\
  \nabla \gz^{-1}\cdot\frac{\xi(\gz^{-1})\gz^{-1}}{|\gz^{-1}|^2}& \bigg\|_{L^\infty} |x| 
\end{split}
\]
Note that the first factor is bounded in the $L^\infty$ norm due to the assumptions on $k$ and $\varphi$. 
As for the second factor, since $\nabla\gz^{-1}$ is bounded and since $|\gz^{-1}(y)|$ is comparable to $|y|$, 
and ${|\xi(y) y|}/{|y|^2}$ is dominated by $C\|\xi\|_X$, we have 
$\| \nabla \gz^{-1}\cdot {\xi(\gz^{-1})\gz^{-1}}/{|\gz^{-1}|^2} \|_{L^\infty} \leq C \|\xi\|_X$. 

For \eqref{eq: 69}, since $|\partial_i \gz(x)\cdot ({\xi(x)}/{|x|^2})x|\leq C \|\xi\|_X|x|$, it is enough to have 
 $\nabla^2 L^{-1}(\rho_0(\gz^{-1}) k(\varphi(\gz^{-1})))(z)$ be bounded and continuous. 
 With our assumptions on $\zeta$ and $\varphi$, $\rho_0(\gz^{-1}) k(\varphi(\gz^{-1}))$ 
 is in $C^1$ and therefore \eqref{eq: 69} satisfies the desired estimate due to Lemma \ref{lem: D2 L inv}. 
 As for \eqref{eq: 70}, since $|\partial_i\left(\frac{\xi(x)}{|x|^2}x\right)|\leq C\|\xi\|_X$, 
 it is enough to have $|\nabla L^{-1}(\rho_0(\gz^{-1}) k(\varphi(\gz^{-1})))(z)|\leq C|z|$, 
 which  is indeed the case by Lemma \ref{lem: 1.1}, using the 
 fact that $\rho_0(\gz^{-1}) k(\varphi(\gz^{-1}))$ is in $L^{\infty}$.  
 
Finally we examine the spatial derivative $\frac{\pa \F_3}{\pa \varphi}  (\zeta,\varphi)$: 
\begin{align}
&~\pa_i\left[\frac{\partial \F_3}{\partial \varphi}(\zeta,\varphi)\eta\right](x) \notag\\
=&~~ \nabla L^{-1} \left[ \rho_0(\gz^{-1})k'(\varphi(\gz^{-1})) \eta (\gz^{-1}) \right] (\gz(x)) \cdot \pa_i \gz(x) . 
\end{align}
Since $\| \rho_0(\gz^{-1})k'(\varphi(\gz^{-1})) \eta (\gz^{-1})\|_{L^\infty} \leq C \| \eta\|_X$, 
from Lemma \ref{lem: 1.1} we deduce that $\|\frac{\pa \F_3}{\pa \varphi}  (\zeta,\varphi)\eta\|_X\leq C \|\eta\|_X|x|$. 
This completes the proof of the lemma. 
\end{proof}

\subsection{Continuity of $\F'$}

Now we establish the continuity of the formal derivative 
$\frac{\partial \F}{\partial (\zeta,\varphi,\omega,\epsilon)}$ from  
$X_\delta\times X_\delta\times (-\delta,\delta)\times (-\delta,\delta)$ into $X\times X$. 

\begin{lemma}  
\label{lem: continuity Frechet}
If $\zeta$, $\varphi\in X_\delta$ and $\delta$ is sufficiently small, 
there exist constants $C>0$, $0<\alpha<1$, and a constant $C_\beta$ for each $0<\beta<1$ such that 
\begin{align}
\left\| \left( \frac{\pa \F_1}{\pa \zeta}  (\zeta_1,\varphi_1) -  \frac{\pa \F_1}{\pa \zeta}  (\zeta_2,\varphi_2)\right)\xi  \right\|_X&\leq C \|\zeta_1-\zeta_2\|_X^\alpha \|\xi\|_X,\label{eq: 29}  \\
\left\| \left( \frac{\pa \F_1}{\pa \varphi}  (\zeta_1,\varphi_1) - \frac{\pa \F_1}{\pa \varphi}  (\zeta_2,\varphi_2) \right) \eta  \right\|_X &\leq C \epsilon \|\varphi_1-\varphi_2\|_X  \|\eta \|_X, \label{eq: 30}\\
\left\| \left( \frac{\pa \F_2}{\pa \zeta}  (\zeta_1,\varphi_1) -  \frac{\pa \F_2}{\pa \zeta}  (\zeta_2,\varphi_2)\right) \xi  \right \|_X &\leq C_\beta \epsilon 
(\|\zeta_1-\zeta_2\|_X^\beta \label{eq: 31}\\
&\qquad \quad+\|\varphi_1-\varphi_2\|_X^\beta) \|\xi\|_X, \notag\\
\left\| \left( \frac{\pa \F_2}{\pa \varphi}  (\zeta_1,\varphi_1)- \frac{\pa \F_2}{\pa \varphi}  (\zeta_2,\varphi_2)\right) \eta  \right\|_X &\leq C_\beta\epsilon 
(\|\zeta_1-\zeta_2\|_X^\beta \label{eq: 32}\\
&\qquad \quad +\|\varphi_1-\varphi_2\|_X ) \|\eta \|_X, \notag
\end{align}
\end{lemma}

\begin{proof}  The inequality \eqref{eq: 29} was derived in Lemma 5.6 of \cite{SW2017}.  
We focus on the rest of the estimates, beginning with \eqref{eq: 30}. Recalling \eqref{eq: 17}, we have
\[
\begin{split}
&\pa_i \left[ \left( \frac{\pa \F_1}{\pa \varphi}  (\zeta_1,\varphi_1) - \frac{\pa \F_1}{\pa \varphi}  (\zeta_2,\varphi_2) \right) \eta\right] (x) \\
&= \epsilon( K''(\varphi_2(x)) \pa_i \varphi_2(x) -K''(\varphi_1(x)) \pa_i\varphi_1(x) ) \eta(x)  \\
&\quad+ \epsilon ( K'(\varphi_2(x)) - K'(\varphi_1(x)) )\pa_i \eta (x)
\end{split}
\] 
The second term in the right-hand side is bounded by 
$C\epsilon \|\varphi_1-\varphi_2\|_X \|\eta\|_X |x|$ 
since $|K'(\varphi_2(x)) - K'(\varphi_1(x))|$ is bounded by $ C\|\varphi_1-\varphi_2\|_X$. 
Rewriting the first term as 
$$  \epsilon( K''(\varphi_2(x)) (\pa_i \varphi_2(x) -\pa_i \varphi_1(x) ) 
+( K''(\varphi_2(x)) -K''(\varphi_1(x))) \pa_i\varphi_1(x) ) \eta(x), $$ 
we see that it is also bounded by $C\epsilon \|\varphi_1-\varphi_2\|_X \|\eta\|_X |x|$. Thus \eqref{eq: 30} holds.

As for \eqref{eq: 31}, we write 
\[
\left( \frac{\pa \F_2}{\pa \zeta}  (\zeta_1,\varphi_1) -  \frac{\pa \F_2}{\pa \zeta}  (\zeta_2,\varphi_2)\right) \xi  = J_1 + J_2 + J_3 +J_4,
\]
where 
\[ 
\begin{split}
J_1&= \epsilon [\M' (\zeta_2)\xi - \M'(\zeta_1)\xi ] \F_3(\zeta_1,\varphi_1) \\
J_2&= \epsilon [\M'(\zeta_2)\xi](\F_3(\zeta_2,\varphi_2)- \F_3(\zeta_1,\varphi_1)) \\
J_3&=\epsilon (\M(\zeta_2) - \M (\zeta_1)) \frac{\pa \F_3}{\pa \zeta}  (\zeta_1,\varphi_1) \xi \\
J_4&= \epsilon \M(\zeta_2)\left[  \frac{\pa \F_3}{\pa \zeta}  (\zeta_2,\varphi_2) \xi  -\frac{\pa \F_3}{\pa \zeta}  (\zeta_1,\varphi_1) \xi  \right]
\end{split}
\]
We will estimate the $X$ norm of each $J_i$, $i=1,2,3,4$. 
We start with $J_1$. The $X$ norm of $\F_3$ was shown to be bounded in Lemma \ref{lem: 2.1}.   
The estimate of $I_1$ in Lemma 5.6 of \cite{SW2017} 
shows $\|\M' (\zeta_2)\xi - \M'(\zeta_1)\xi \|_{L^\infty} \leq C \|\zeta_1-\zeta_2\|_X\|\xi \|_X$, 
and therefore we obtain $$\| J_1 \|_X \leq C \epsilon \|\zeta_1-\zeta_2\|_X\|\xi \|_X.$$ 

We next estimate $J_2$. By Lemma 5.5 of \cite{SW2017}, we have $\| \M'(\zeta_2)\xi \|_X\leq C \|\xi\|_X$. 
In order to estimate the $X$ norm of $\F_3(\zeta_2,\varphi_2)- \F_3(\zeta_1,\varphi_1)$, 
we rewrite its spatial derivative as 
\[
\begin{split}
&\pa_i ( \F_3(\zeta_2,\varphi_2)- \F_3(\zeta_1,\varphi_1) ) =\\
&  \quad
\{ \nabla L^{-1}[ \rho_0(\gztwo^{-1}) k(\varphi_2 (\gztwo^{-1}))  ](\gztwo(x)) - 
\nabla L^{-1}[ \rho_0(\gzone^{-1}) k(\varphi_2 (\gzone^{-1}))  ](\gztwo(x)) \} \pa_i \gztwo (x) \\
&+ \{ \nabla L^{-1}[ \rho_0(\gzone^{-1}) k(\varphi_2 (\gzone^{-1}))  ] (\gztwo(x)) - 
\nabla L^{-1}[ \rho_0(\gzone^{-1}) k(\varphi_2 (\gzone^{-1}))  ] (\gzone(x)) \} \pa_i \gztwo (x)\\
&+ \{ \nabla L^{-1}[ \rho_0(\gzone^{-1}) k(\varphi_2 (\gzone^{-1}))  ] (\gzone(x)) - 
\nabla L^{-1}[ \rho_0(\gzone^{-1}) k(\varphi_1 (\gzone^{-1}))  ] (\gzone(x)) \}\pa_i \gztwo (x)\\
&+ \nabla L^{-1}[ \rho_0(\gzone^{-1}) k(\varphi_1 (\gzone^{-1}))  ] (\gzone(x)) \pa_i  (\gztwo(x) -  \gzone(x) ) \\
&=: J_{21} + J_{22} + J_{23}+ J_{24}
\end{split}
\]
$J_{21}$ can be estimated 
by the representation \eqref{eq: 290} and Lemma \ref{lem: 1.5} with $u= \rho_0 k (\varphi_2)$. 
Since $ \rho_0 k (\varphi_2)$ is in $L^\infty$, we deduce that 
$
|J_{21}| \leq C_\beta \|\zeta_1 - \zeta_2\|_X^\beta |x|$ for any $0<\beta<1$. 
For $J_{22}$, we apply Lemma \ref{lem: 1.2} to deduce that 
$
|J_{22}| \leq C_\beta \|\zeta_1 - \zeta_2\|_X^\beta  |x| 
$ 
 for any $0<\beta<1$, as $\rho_0(\gzone^{-1})k(\varphi_2(\gzone^{-1}))\in C^1$.  
To estimate $J_{23}$, we note that 
$|k(\varphi_2(z)) -  k( \varphi_1(z))|\leq \|k'\|_{L^\infty} |\varphi_2(z)- \varphi_1 (z)| \leq C \|\varphi_1-\varphi_2\|_X$, 
which results in 
$
|J_{23}|\leq C \| \varphi_1-\varphi_2 \|_X |x|$ by Lemma \ref{lem: 1.2}.   
Moreover, we have $|J_{24}|\leq C\| \zeta_1 -\zeta_2\|_X|x|$, 
since $|\pa_i  (\gztwo(x) -  \gzone(x) )|\leq C\| \zeta_1 -\zeta_2\|_X$.  
To sum up, we deduce that 
$$\|J_2\|_X\leq C\epsilon (   \|\zeta_1 - \zeta_2\|_X^\beta+   \|\varphi_1-\varphi_2\|_X ) \|\xi\|_X $$ 
for any $0<\beta<1$. 

The estimation of $J_3$ is similar to that of $J_1$.  
From Lemma \ref{lem: 2.2} we deduce that $$\| J_3 \|_X \leq C \epsilon \|\zeta_1-\zeta_2\|_X\|\xi \|_X.$$

We now turn into $J_4$.  
Introducing some notation to denote the functions appearing in  \eqref{eq: 68}, \eqref{eq: 69}, \eqref{eq: 70}, 
we define 
\[
\begin{split}
v(\zeta,\varphi)&:= \left(\nabla\rho_0(\gz^{-1})k(\varphi(\gz^{-1}))+\rho_0(\gz^{-1})k'(\varphi(\gz^{-1}))\nabla \varphi(\gz^{-1})\right)\cdot  D \gz^{-1}\cdot\frac{\xi(\gz^{-1})\gz^{-1}}{|\gz^{-1}|^2} \\
w(\zeta,\varphi)&:= \rho_0(\gz^{-1}) k(\varphi(\gz^{-1}))
\end{split}
\]
Then we may write 
$\pa_i\left[  \frac{\pa \F_3}{\pa \zeta}  (\zeta_2,\varphi_2) \xi  -\frac{\pa \F_3}{\pa \zeta}  (\zeta_1,\varphi_1) \xi  \right]$ as   
\[
\begin{split}
&\;\;\pa_i\left[  \frac{\pa \F_3}{\pa \zeta}  (\zeta_2,\varphi_2) \xi  -\frac{\pa \F_3}{\pa \zeta}  (\zeta_1,\varphi_1) \xi  \right] \\
&=\nabla L^{-1}(v(\zeta_1,\varphi_1))(\gzone(x))\cdot \pa_i \gzone(x) -  \nabla L^{-1}(v(\zeta_2,\varphi_2))(\gztwo(x))\cdot \pa_i \gztwo(x) \\
&+ \left[\nabla^2 L^{-1} (w(\zeta_2,\varphi_2)) (\gztwo(x)) \pa_i \gztwo(x) -  \nabla^2 L^{-1} (w(\zeta_1,\varphi_1)) (\gzone(x)) \pa_i \gzone(x)\right] \cdot \frac{\xi(x)}{|x|^2}x\\
&+\left[ \nabla L^{-1} (w(\zeta_2,\varphi_2)) (\gztwo(x)) -\nabla L^{-1} (w(\zeta_1,\varphi_1)) (\gzone(x))  \right]\cdot \pa_i \left( \frac{\xi(x)}{|x|^2}x\right) \\
&=: J_{41} + J_{42} + J_{43}. 
\end{split}
\]
We will estimate each $J_{4i}$, $i=1,2,3$ separately.   
We start with $J_{41}$ following the same steps as  for $J_2$.  
Namely, we rewrite it as $J_{411}+ J_{412}+ J_{413}+ J_{414}$ (see $J_{21} + J_{22} + J_{23}+ J_{24}$) 
and estimate each term. The only novelty is that we have $v(\zeta,\varphi)$ 
instead of $ w(\zeta,\varphi)=\rho_0(\gz^{-1}) k(\varphi(\gz^{-1}))$. 
Consequently, we just use estimate \eqref{est: lem 1.5.2} instead of \eqref{est: lem 1.5.1}. 
We deduce that 
$$|J_{41}| \leq C_\beta(   \|\zeta_1 - \zeta_2\|_X^\beta+   \|\varphi_1-\varphi_2\|_X ) \|\xi\|_X|x| $$ for any $0<\beta<1$. 

For $J_{42}$, we rewrite the square bracket term as 
\[
\begin{split}
 [ \ \cdot \ ] &=\nabla^2 L^{-1} [ w(\zeta_2,\varphi_2) - w(\zeta_1,\varphi_2)] (\gztwo(x)) \pa_i \gztwo(x) \\
  &+  \nabla^2 L^{-1} [ w(\zeta_1,\varphi_2) - w(\zeta_1,\varphi_1)] (\gztwo(x)) \pa_i \gztwo(x) \\
  &+ \{ \nabla^2 L^{-1} [w(\zeta_1,\varphi_1))] (\gztwo(x)) - \nabla^2 L^{-1}[ w(\zeta_1,\varphi_1)] (\gzone(x))\} \pa_i \gztwo(x) \\
  &+ \nabla^2 L^{-1} [w(\zeta_1,\varphi_1)] (\gzone(x)) \pa_i ( \gztwo(x)  -\gzone(x)  ) \\
  &=: J_{421} +  J_{422} +  J_{423} +  J_{424}
\end{split}
\]
We focus first on $\nabla^2 L_1^{-1} [ w(\zeta_2,\varphi_2) - w(\zeta_1,\varphi_2)](z)$ in $J_{421}$. 
Since $\rho_0 k(\varphi_2) \in C^{0,1}$, Lemma \ref{lem: 1.5} yields  
$\| w(\zeta_2,\varphi_2) - w(\zeta_1,\varphi_2) \|_{C^{0,\alpha}} \leq C_\alpha \| \zeta_1-\zeta_2 \|_X^{1-\alpha}$ for any $0<\alpha<1$. Lemma \ref{lem: D2 L inv} then gives the bound   
$$
\|\nabla^2 L_1^{-1} ( w(\zeta_2,\varphi_2) - w(\zeta_1,\varphi_2))\|_{L^\infty}\leq C_\beta \| \zeta_1-\zeta_2 \|_X^{\beta}$$ 
for any  $0<\beta<1$. Similarly, in $J_{422}$ we also have 
$$\|\nabla^2 L_1^{-1} [w(\zeta_1,\varphi_2) - w(\zeta_1,\varphi_1)]\|_{L^\infty}\leq C_\beta \| \varphi_1-\varphi_2 \|_X^{\beta}$$ 
 for any  $0<\beta<1$.   
 The most serious term in $J_{423}$ is as follows. 
 Since $\rho_0 k(\varphi_1)\in C^{0,1}$, Lemma \ref{lem: D2 L inv} again gives 
\begin{align*}
&~|  \nabla^2 L_1^{-1} [w( \zeta_1,\varphi_1) ](\gztwo(x)) - (\nabla^2 L_1^{-1} [w( \zeta_1,\varphi_1) ] (\gzone(x))| \\
\leq &~C_\beta | \gzone(x) - \gztwo(x)|^\beta \leq C_\beta \|\zeta_1-\zeta_2\|_X^\beta
\end{align*}
for any $0<\beta<1$. Lastly, $|J_{424}|\leq C \| \zeta_1 -\zeta_2\|_X$ since $\nabla^2 L^{-1} [w(\zeta_1,\varphi_1)] (\gzone(x))$ is bounded and $ |\pa_i  (\gztwo(x) -  \gzone(x) )|\leq C\| \zeta_1 -\zeta_2\|_X$. Together with $ |\frac{\xi(x)}{|x|^2}x|\leq C\|\xi\|_X|x|$, we arrive at 
\[
|J_{42}|\leq C_\beta  (   \|\zeta_1 - \zeta_2\|_X^\beta+   \|\varphi_1-\varphi_2\|^\beta_X ) \|\xi\|_X|x|
\] 
for any $0<\beta<1$. 

The square bracket in $J_{43}$ can be  estimated in the same way as was done for 
$J_{21}$, $J_{22}$ and $J_{23}$. Together with $|\pa_i \left( \frac{\xi(x)}{|x|^2}x \right)|\leq C \|\xi\|_X$, we have  
\[
|J_{43}| \leq C_\beta (   \|\zeta_1 - \zeta_2\|_X^\beta+   \|\varphi_1-\varphi_2\|_X ) \|\xi\|_X|x| 
\]
 for any $0<\beta<1$. Combining the estimates of $J_{41},$  $J_{42}$ and  $J_{43}$, we find that 
 \[
 \|J_4\|_X \leq C_\beta \epsilon ( \|\zeta_1 - \zeta_2\|_X^\beta+   \|\varphi_1-\varphi_2\|^\beta_X) \|\xi\|_X
 \]
for any $0<\beta<1$. This completes the proof of \eqref{eq: 31}.

It remains to show \eqref{eq: 32}. Note that 
\[
\begin{split}
\left( \frac{\pa \F_2}{\pa \varphi}  (\zeta_1,\varphi_1)- \frac{\pa \F_2}{\pa \varphi}  (\zeta_2,\varphi_2)\right) \eta = \epsilon (\M(\zeta_2) - \M (\zeta_1)) \frac{\pa \F_3}{\pa \varphi}  (\zeta_1,\varphi_1) \eta \\
+ \epsilon \M(\zeta_2)\left[  \frac{\pa \F_3}{\pa \varphi}  (\zeta_2,\varphi_2) \eta  -\frac{\pa \F_3}{\pa \varphi}  (\zeta_1,\varphi_1) \eta \right] =: K_1+ K_2
\end{split}
\]
Each factor of $K_1$ has been previously estimated and it is easy to see that 
\[
\|K_1\|_X \leq C\epsilon \|\zeta_1 - \zeta_2\|_X \|\eta\|_X. 
\]
The other term $K_2$ can be estimated in the same spirit as for $J_2$, for instance.  We obtain   
\[
\| K_2\|_X \leq C_\beta \epsilon (   \|\zeta_1 - \zeta_2\|_X^\beta+   \|\varphi_1-\varphi_2\|_X ) \|\eta\|_X 
\]
for any $0<\beta<1$. We omit the details. This completes the proof of the proposition. 
\end{proof}

Before proving that the formal derivatives are genuine Fr\'echet derivatives, 
we need a technical lemma that establishes the equality of the mixed partial derivatives of the functions
$$ G_1(x,s)=\F(\zeta+s\xi, \varphi)(x) 
=(\F_1(\zeta+s\xi, \varphi)(x),\F_2(\zeta+s\xi, \varphi)(x))^t  $$
and
$$ G_2(x,s) = \F(\zeta,\varphi+s\eta)(x) 
=(\F_1(\zeta,\varphi+s\eta)(x),\F_2(\zeta,\varphi+s\eta)(x))^t  $$ 
This will of course be true if the $G_1$, $G_2$ are $C^2$. However, rather than showing such regularity, we make a direct computation.

\begin{lemma}\label{lem: mixed partial}
Let $G_1$, $G_2$ be defined as above, where $\zeta,\varphi\in X_\delta$ for $\delta>0$ sufficiently small, 
and $\xi,\eta\in X$ are such that $\zeta+s\xi,\varphi+s\eta\in X_\delta$ for all $|s|\leq 1$. 
Then 
\eqn\label{eq: mixed partial}
\partial_i\partial_s G_j(x,s)=\partial_s\partial_iG_j(x,s)
\eeqn
for all $x\in \overline{B_1}$, all $|s|\le 1$ and $j=1,2$.
\end{lemma}

\begin{proof}
We only show the proof for $G_1$, as $G_2$ is similar.  
For the $\F_1$ part of $G_1$, \eqref{eq: mixed partial} is an immediate consequence 
of Lemma 5.4 of \cite{SW2017}.   For $\F_2$, we only focus on the part 
\eqn
G_3(x,s)=\F_3(\zeta+s\xi,\varphi)(x)=L^{-1}[\rho_0(g_{\zeta+s\xi}^{-1})k(\varphi(g_{\zeta+s\xi}^{-1}))](g_{\zeta+s\xi}(x)).
\eeqn
In the following, to simplify notation, we plug in $s=0$ after taking the derivative $\partial_s$. The general case will be similar but has more clumsy notation. $\partial_i\partial_sG_3(x,0)$ is given by \eqref{eq: 68}, \eqref{eq: 69} and \eqref{eq: 70}. To find the other mixed partial derivative, we compute 
\begin{align}
\partial_i G_3(x,s) = \nabla L^{-1}[\rho_0(g_{\zeta+s\xi}^{-1})k(\varphi(g_{\zeta+s\xi}^{-1}))](g_{\zeta+s\xi}(x))\cdot \partial_ig_{\zeta+s\xi}(x).
\end{align}
If we let the $\partial_s$ derivative fall on the various terms as follows
\begin{align}
&~\nabla L^{-1}\{\partial_s[\rho_0(g_{\zeta+s\xi}^{-1})k(\varphi(g_{\zeta+s\xi}^{-1}))]\}(g_{\zeta+s\xi}(x))\cdot \partial_ig_{\zeta+s\xi}(x)\\
&~+\nabla^2 L^{-1}[\rho_0(g_{\zeta+s\xi}^{-1})k(\varphi(g_{\zeta+s\xi}^{-1}))](g_{\zeta+s\xi}(x))\cdot \partial_s G_{\zeta+s\xi}(x)\cdot \partial_ig_{\zeta+s\xi}(x)\\
&~+\nabla L^{-1}[\rho_0(g_{\zeta+s\xi}^{-1})k(\varphi(g_{\zeta+s\xi}^{-1}))](g_{\zeta+s\xi}(x))\cdot\partial_s \partial_ig_{\zeta+s\xi}(x),
\end{align}
we will recover \eqref{eq: 68}, \eqref{eq: 69} and \eqref{eq: 70}, and the proof will be complete. 
To justify the calculation above, we define
\eqn
G_4(z,s) = \nabla L^{-1}[\rho_0(g_{\zeta+s\xi}^{-1})k(\varphi(g_{\zeta+s\xi}^{-1}))](z)
\eeqn
and want to show that 
\begin{enumerate}[(a)]
\item $\partial_sG_4(z,s) = \nabla L^{-1}\{\partial_s[\rho_0(g_{\zeta+s\xi}^{-1})k(\varphi(g_{\zeta+s\xi}^{-1}))]\}(z)$.
\item $\partial_zG_4(z,s) = \nabla^2 L^{-1}[\rho_0(g_{\zeta+s\xi}^{-1})k(\varphi(g_{\zeta+s\xi}^{-1}))](z)$.
\item Both $\partial_sG_4(z,s)$ and $\partial_zG_4(z,s)$ are continuous. \end{enumerate}
The calculation will therefore be justified by the chain rule. 

Since $\partial_s[\rho_0(g_{\zeta+s\xi}^{-1})k(\varphi(g_{\zeta+s\xi}^{-1}))]$ is bounded, (a) follows directly from the dominated convergence theorem. (b) is obvious. The continuity of $\partial_sG_4(z,s)$ and $\partial_zG_4(z,s)$ can be proven by the same kind of estimates employed in the proof of Lemma \ref{lem: continuity Frechet}. 
We omit the straightforward details.
\end{proof}

We finally show that the formal derivatives are genuine Fr\'{e}chet derivatives. 
\begin{lemma}\label{lem: Frechet}
Let $\zeta,\varphi\in X_\delta$ where $\delta>0$ is sufficiently small.  
There exist $\delta_1>0$, $0<\alpha<1$ and $C>0$ such that if $\|\xi\|_X,\|\eta\|_X<\delta_1$, then 
\eqn  
\label{eq: lem 2.5.1}
\left\|\F(\zeta+\xi,\varphi+\eta)-\F(\zeta,\varphi)-\frac{\partial \F}{\partial(\zeta,\varphi)}(\xi,\eta)\right\|_{X^2}   
\le C(\|\xi\|_X+\|\eta\|_X)^{1+\alpha},
\eeqn 
where $\frac{\partial \F}{\partial(\zeta,\varphi)}$ denotes the formal derivative defined in \eqref{def: formal deriv}.
\end{lemma}
\begin{proof}
In order to estimate the $X$ norm on the left hand side of \eqref{eq: lem 2.5.1}, we compute the spatial derivatives by 
\begin{align}
&~\partial_i[\F(\zeta+\xi,\varphi+\eta)-\F(\zeta,\varphi)](x) \notag\\
=&~ \partial_i\F(\zeta+s\xi,\varphi+s\eta)(x)\bigg|_{s=0}^{s=1}\notag\\
=&~ \partial_s\partial_i\F(\zeta+s\xi,\varphi+s\eta)(x)\bigg|_{s=\theta(x)}\notag\\
=&~\partial_i\partial_s\F(\zeta+s\xi,\varphi+s\eta)(x)\bigg|_{s=\theta(x)}\notag\\
=&~\partial_i \left[\left(\frac{\partial\F}{\partial (\zeta,\varphi)}(\zeta+s\xi,\varphi+s\eta)\right)(\xi,\eta)\right](x)\bigg|_{s=\theta(x)},
\end{align}
where $0<\theta(x)<1$.   
We have used Lemma \ref{lem: mixed partial} to exchange the order of mixed partial derivatives. 
It follows that
\begin{align}
&~ \partial_i\left[\F(\zeta+\xi,\varphi+\eta)-\F(\zeta,\varphi)-\frac{\partial \F}{\partial(\zeta,\varphi)}(\xi,\eta)\right](x)\label{eq: lem 2.5.2}\\
=&~\partial_i\left[\left(\frac{\partial\F}{\partial (\zeta,\varphi)}(\zeta+s\xi,\varphi+s\eta)-\frac{\partial \F}{\partial(\zeta,\varphi)}(\zeta,\varphi)\right)(\xi,\eta)\right](x)\bigg|_{s=\theta(x)}. \notag
\end{align}
Using Lemma \ref{lem: continuity Frechet}, the components of \eqref{eq: lem 2.5.2} are bounded by
\begin{align*}
&~C(\|s\xi\|_X+\|s\eta\|_X)^\alpha(\|\xi\|_X+\|\eta\|_X)|x|\bigg|_{s=\theta(x)}\\
\le &~ C(\|\xi\|_X+\|\eta\|_X)^{1+\alpha}|x|.
\end{align*}
The estimate \eqref{eq: lem 2.5.1} thus follows.
\end{proof}

Since the dependence of $\F$ on $\omega$ and $\epsilon$ is very simple, 
we easily see that $\F$ is continuously Fr\'echet differentiable with respect to these two variables as well. 
Thus Theorem \ref{Frechet theorem} immediately follows from Lemma \ref{lem: Frechet}.

\

\noindent{\bf Acknowledgments.}  
JJ is supported in part by NSF grant DMS-1608494. 
YW is supported in part by NSF grant DMS-1714343.  
We acknowledge the support of the spring 2017 semester program at ICERM (Brown U.) where this work was begun. We would also like to thank the anonymous referee for valuable comments which have improved the presentation of the paper.

\end{document}